\documentclass[12pt]{amsart}
\usepackage{mathrsfs}
\usepackage[colorlinks]{hyperref}
\usepackage{xcolor}
\usepackage[a4paper,asymmetric]{geometry}
\usepackage{mathscinet}
\usepackage{fullpage}
\usepackage{latexsym}
\usepackage{amsthm}
\usepackage{amssymb}
\usepackage{amsfonts}
\usepackage{amsmath}
\newtheorem{theorem}{Theorem}[section]
\newtheorem{thm}[theorem]{Theorem}

\newtheorem{lem}[theorem]{Lemma}
\newtheorem{proposition}[theorem]{Proposition}

\newtheorem{corollary}[theorem]{Corollary}

\theoremstyle{definition}

\newtheorem{defn}[theorem]{Definition}

\theoremstyle{remark}

\newtheorem{rem}[theorem]{Remark}
\numberwithin{equation}{section}

 \DeclareMathAlphabet{\mathpzc}{OT1}{pzc}{m}{it}
 
 \newcommand{\E}{\mathbb{E}}            

 \newcommand{\p}{\partial}

 \newcommand{\R}{\mathbb{R}}

 \newcommand{\mcl}{\mathcal}
 
 \newcommand{\Be}{\begin{equation}}
 \newcommand{\Ee}{\end{equation}}
 \newcommand{\Bs}{\begin{split}}
 \newcommand{\Es}{\end{split}}
  \newcommand{\Bes}{\begin{equation*}}
 \newcommand{\Ees}{\end{equation*}}
 \newcommand{\BT}{\begin{thm}}
 \newcommand{\ET}{\end{thm}}
 \newcommand{\dif}{\mathrm{d}}
 \newcommand{\Bp}{\begin{proof}}
 \newcommand{\Ep}{\end{proof}}
 \newcommand{\BL}{\begin{lem}}
 \newcommand{\EL}{\end{lem}}
 \newcommand{\BP}{\begin{proposition}}
 \newcommand{\EP}{\end{proposition}}
 \newcommand{\BC}{\begin{corollary}}
 \newcommand{\EC}{\end{corollary}}
 \newcommand{\BR}{\begin{rem}}
 \newcommand{\ER}{\end{rem}}
 \newcommand{\BD}{\begin{defn}}
 \newcommand{\ED}{\end{defn}}
 \newcommand{\BI}{\begin{itemize}}
 \newcommand{\EI}{\end{itemize}}
 
 \newcommand{\tl}{\tilde}

\begin{document}
\title
[On some  smoothening effects of the transition semigroup  of  a  L\'evy process]
{On some  smoothening effects of the transition semigroup  of  a  L\'evy process}

\author{Zhao Dong}
\address{Institute of Applied Mathematics, Academy of Mathematics and Systems Sciences, Academia Sinica, P.R.China}
\email{dzhao@amt.ac.cn}

\author{Szymon Peszat}
\address{Institute of Mathematics, Polish Academy of Sciences, \'Sw. Tomasza 30/7, 31-027 Cracow, Poland}
\email{napeszat@cyf-kr.edu.pl}

\author{Lihu Xu}
\address{Department of Mathematics, Faculty of Science and Technology University of Macau, Av. Padre Tom\'{a}s Pereira, Taipa Macau, China}
\email{lihuxu@umac.mo}

 \thanks{ZD is supported by Key Laboratory of Random Complex Structures and Data Science, Academy of Mathematics and Systems Science, CAS (No.2008DP173182), NSFC(No. 11271356), NSFC(No. 10721101), 973 Program (2011CB808000). LX is supported by the grant SRG2013-00064-FST.  SP is supported by Polish National Science Center grant  DEC2013/09/B/ST1/03658.  
 The authors would like to gratefully thank Prof. Zhen-Qing Chen and Dr. Jieming Wang for stimulating discussions, and  Profs. Jiahong Wu and Zhifei Zhang for their helpful discussions on fractional Laplacian and Campanato's theorem.}

\begin{abstract}
Let $(P_t)$ be the transition semigroup of a L\'evy process $L$ taking values in a Hilbert space $H$. Let $\nu$ be the L\'evy measure of $L$. It is shown that for any bounded and measurable function $f$, 
$$
\int_H\left\vert P_tf(x+y)-P_tf(x)\right\vert ^2 \nu (\dif y)\le \frac 1 t P_tf^2(x) \qquad \text{for all $t>0$, $x\in H$.}  
$$
As  $\nu$ can be infinite this formula establishes some smoothening effect of the semigroup $(P_t)$.  In the paper some applications of the formula  will be presented  as well. 
\end{abstract}

\maketitle

\section{Introduction}
Let $(X_t)$ be the solution to an SDE on a Hilbert space $H$   driven by a non-degenerate Wiener process $W$.  Let   
$$
P_tf(x)= \mathbb{E}\left( f(X_t)\vert X_0=x\right), \qquad  f\in B_b(H), \ t\ge 0, 
$$ 
be the corresponding transition semigroup defined on the space of bounded measurable functions $B_b(H)$. Then   the following   Bismut--Elworthy--Li formula holds (see \cite{ElLi94} or \cite{PeZa95})
$$
\langle \nabla P_tf(x),h\rangle _H=\frac 1t\mathbb{E}\left( f(X_t) \int_0^t K(s;h) \dif W_s\vert X_0=x\right), 
$$
where $K(s,h)$ is an adapted stochastic processes independent of $f$. This formula implies the strong Feller property of $(P_t)$, and therefore is very useful for studying its ergodic properties. For other applications we refers readers to  e.g.  \cite{WaXu12, BWY13, EcHa01}. 

In this paper, we prove  that the transition semigroup  $(P_t)$ of a L\'evy process $L$ taking values in a Hilbert space $H$ exhibits some smoothening effects, namely $P_t$ transforms $B_b(H)$ into the intersection of domains of some non-local operators (see Theorem \ref{T31} for more details). The proof of Theorem \ref{T31}   is very simple and follows  \cite{ElLi94}. Next, see  Corollary \ref{C32}, we will show that 
\begin{equation}\label{E11}
\int_H\left\vert P_tf(x+y)-P_tf(x)\right\vert ^2 \nu (\dif y)\le \frac 1 t P_tf^2(x) \quad \text{for  $f\in B_b(H)$,  $t>0$, $x\in H$.}  
\end{equation}
Above  $\nu $ is the L\'evy measure of $L$. Note that if $\nu$ is infinite, then for any open ball $B_\varepsilon(0)$ with the center at $0$ and radius  $\varepsilon >0$ one has $\nu\left(B_\varepsilon(0)\right)=+\infty$.  Therefore  \eqref{E11} means that for any $x\in H$,  $P_tf(x+y)$, $y\in B_\varepsilon(0)$, is in a certain sense close to $P_tf(x)$.  

As applications of our general result, we obtain a short time estimate for the semigroup generated by fractional Laplacian and a `fractional gradient' estimate for the perturbed stable type stochastic systems considered in  \cite{ChWa12}. Finally using generalized Campanato's theorem, we calculate modulus of continuity of the transition semigroups of  '$\log$ $\alpha$-stable' processes.

The paper is organized as  follows: the next section includes some preliminary facts on  L\'evy processes. The main general results are formulated in Section \ref{SMain}. The last three sections are devoted to applications.  In  the appendix we  prove  the generalized Campanato theorem of  harmonic analysis.

\section{Preliminary facts on L\'evy processes} 
We shall recall here some preliminary facts on L\'evy processes (for details see e.g. \cite{Ap04, PeZa07,Sa90}). Let $(L_t)_{t \ge 0}$ be an $H$-valued L\'evy process.  It is well known that there is a vector $m\in H$, a symmetric positive definite trace class operator $Q\colon H\mapsto H$, and a Borel measure $\nu$ on $H$ satisfying  
\begin{equation}\label{E21}
\nu(\{0\})=0, \qquad \int_{H}1\wedge \vert y\vert _H^2\nu( \dif y)<+\infty, 
\end{equation}
such that 
$$
\mathbb{E}\,  {\mathrm e}^{\mathrm{i}\langle x,   L_t\rangle_H}={\mathrm e}^{-t\psi(x)},\qquad x\in H, 
$$
where the so-called \emph{L\'evy symbol}  $\psi$ of $(L_t)$ is given by
$$
\psi(x)=i\langle x, m\rangle _H+ \frac 12 \langle Qx,  x\rangle_H + \int_{H}\left[{\mathrm e}^{\mathrm{i}\langle x, y\rangle_H}-1-\mathrm{i}\langle x,   y\rangle _H\mathbf{1}_{\{\vert y\vert _H\leq 1\}}\right]\nu( \dif y). 
$$
We call $\nu$ the \emph{L\'evy measure} of $L$ and $(m,Q,\nu)$ the \emph{generating triplet} of $L$. 

The \emph{Poisson random measure} associated with $(L_t)$ is defined by
$$
N(t,\Gamma):=\sum_{s\in(0,t]}\mathbf{1}_{\Gamma}(L_s-L_{s-}),\qquad \Gamma\in  \mathcal{B}(H), \ t>0, 
$$
and the \emph{compensated Poisson random measure} is given by
$$
\tilde N(t,\Gamma)=N(t,\Gamma)-t\nu(\Gamma).
$$
By the\emph{ L\'evy--Khinchin decomposition} (cf. \cite[p.108, Theorem 2.4.16]{Ap04} or \cite[p. 53, Theorem 4.23]{PeZa07}), one has
$$
L_t=mt + W_Q(t)+ \int_{\{0<\vert x\vert _H\leq 1\}}x\tilde N(\dif t,\dif x)+\int_{\{\vert x\vert _H>1\}}x N(\dif t,\dif x), \qquad t\ge 0, 
$$
where $W_Q$ is a  Wiener process in $H$ with covariance operator $Q$. 

Let $(\mathfrak{F}_t)$ be the filtration generated by $(L_t)$, and let us  denote by $\mathcal{L}^2_{\mathrm loc}$ the space of all predictable stochastic process $\psi$ satisfying   
$$
\mathbb{E} \int_0^t \int_{H}\vert \psi(s,y)\vert _H^2 \nu(\dif y) \dif s<\infty \qquad \text{for   $t>0$.}
$$
Then for any $\psi\in \mathcal{L}^2_{\mathrm loc}$ the stochastic integral  $\int_0^t \int_{H} \psi(s,y) \tilde N(\dif s, \dif y)$ is a well-defined square integrable and mean zero martingale. Moreover, the following It\^o  isometry holds (see e.g. \cite[p. 200]{Ap04} or \cite[Section 8.7]{PeZa07})
\begin{equation} \label{E22}
\begin{aligned}
&\mathbb{E}\left[\int_0^t \int_{H} \psi(s,y) \tilde  N(\dif s, \dif y) \int_0^t \int_{H} \varphi(s,y) \tilde N(\dif s, \dif y)\right]\\
&\quad =\mathbb{E} \int_0^t \int_{H} \psi(s,y) \varphi(s,y) \nu(\dif y) \dif s, \qquad \psi,\varphi\in \mathcal{L}^2_{\mathrm loc}.
\end{aligned}
\end{equation}

Let $L=(L_t)$ be a L\'evy process with a generating triplet $(m,Q,\nu)$. Consider the Markov family 
\begin{equation}\label{E23}
L^x_t= x+L_t, \qquad t\ge 0, \ x\in H. 
\end{equation}
Its transition semigroup $(P_t)$ is given as follows 
\begin{equation}\label{E24}
P_tf(x)= \mathbb{E}\, f(L^x_t), \qquad t\ge 0, \ x\in H,\ f\in B_b(H). 
\end{equation}
Observe that $(P_t)$ is $C_0$ on the space $UC_b(H)$ of uniformly continuous bounded functions on $H$, see \cite[p. 80]{PeZa07}. Moreover, the domain of its generator $\mathcal L$ contains the space $UC^2_b(H)$, and 
$$
\begin{aligned}
&\mathcal Lf(x)= \langle Df(x),m\rangle _H + \frac 12 \text{\rm Trace}\, QD^2f(x)\\
&\  + \int_{H}\left(f(x+y)-f(x)-\mathbf{1}_{\{\vert x\vert _H<1\}}\langle Df(x), y\rangle_H\right)\nu(\dif y), \qquad f\in UC^2_b(H), \  t\ge 0, \ x\in H. 
\end{aligned}
$$

\section{Main results}\label{SMain}
Let $L$ be a L\'evy process on $H$ with the generation triplet $(m,Q,\nu)$. Let $(L^x_t)$, $t\ge 0$, $x\in H$ be the  Markov family given by \eqref{E23}, and let $(P_t)$ given by \eqref{E24} be the transition semigroup of $(L^x_t)$. 

Given  $q\in L^2(H,\mathcal{B}(H), \nu)$, define 
$$
\mathcal{D}_q:=\left\{f\in B_b(H)\colon \sup_{x\in H}\int_{H} \left\vert f(x+y)-f(x)\right\vert \vert q(y)\vert \nu(\dif y)<\infty\right\}.
$$
Next, let 
$$
A_q f(x): = \int_{H} \left[f(x+y)-f(x)\right]q(y)\nu(\dif y)\qquad  \text{for $f\in \mathcal{D}_q$, $x\in H$.} 
$$
Taking into account  \eqref{E21}, we see that $C^1_b(H)\subset \mathcal{D}_q$, $A_q$ is a bounded linear operator form $C^1_b(H)$ into $B_b(H)$ and 
\begin{align*}
\|A_qf\|_\infty &:= \sup_{x\in H}|A_q f(x)| \\
&\le 2 \|f\|_\infty \left(\int_{|y|_H\ge 1}  \nu(\dif y)\right)^{\frac12}\left(\int_{|y|_H\ge 1} q^2(y) \nu(\dif y)\right)^{\frac12}
\\
&\quad + \|Df\|_\infty \left(\int_{\{|y|_H<1\}}  |y|_H^2 \nu(\dif y)\right)^{\frac12}\left(\int_{\{|y|_H<1\}} q^2(y) \nu(\dif y)\right)^{\frac12} <\infty.
\end{align*}
\begin{theorem}\label{T31}
Let $q\in L^2(H,\mathcal{B}(H), \nu)$. Then for all $t>0$ and $f\in B_b(H)$, 
$P_tf\in \mathcal{D}_q$ and 
$$
\left\vert A_q P_tf(x)\right\vert ^2  \le \frac 1t P_tf^2(x) \int_{H} q^2(y)\nu(\dif y)\qquad \text{for all $x\in H$. }
$$
\end{theorem}
\begin{proof} Assume that $f\in UC^2_b(H)$. Applying It\^{o}'s formula (see e.g. \cite{Ap04} or \cite{PeZa07}) we obtain 
\begin{equation} \label{E31}
\begin{split}
f(L_t^x)-P_t f(x)&=-\int_0^t \mathcal  L P_{t-s} f(L^x_s)\dif s +\int_0^t  P_{t-s} \mathcal  L f(L^x_s) \dif s\\
& \ \ \ +\int_0^t \int_{H}\left[P_{t-s}f(L^x_s+y)-P_{t-s}f(L^x_s))\right] \tilde N(\dif y,\dif s) \\
& \ \ \ +\int_0^t \langle DP_{t-s}f(L^x_s), \dif W_Q(s)\rangle _H\\
&=\int_0^t \int_{H}\left[P_{t-s}f(X_s+y)-P_{t-s}f(X_s)\right] \tilde N(\dif y,\dif s)\\
& \ \ \ +\int_0^t \langle DP_{t-s}f(L^x_s), \dif W_Q(s)\rangle_H, 
\end{split}
\end{equation}
where the last equality is because $\mathcal L P_t f=P_t \mathcal L f$ for $f \in UC^2_b(H)$. Multiplying the both sides of \eqref{E31} by 
$$
\int_0^t \int_{H} q(y) \tilde N(\dif y, \dif s)
$$
and taking into account  \eqref{E22}, we further get
\begin{align*}
\mathbb{E} \left[f(L^x_t) \int_0^t \int_{H} q(y) \tilde  N(\dif y, \dif s)\right] 
&=\mathbb{E}\int_0^t \int_{H} \left[P_{t-s}f(L^x_s+y)-P_{t-s}f(L^x_s)\right] q(y)\nu(\dif y) \dif s\\
&=\int_0^t \int_{H} \left[P_sP_{t-s}f(x+y)-P_s P_{t-s}f(x)\right] q(y)\nu(\dif y) \dif s\\
&=t \int_{H} \left[P_{t}f(x+y)-P_{t}f(x)\right] q(y)\nu(\dif y)= tA_qP_tf(x). 
\end{align*}
Thus, by the H\"older inequality and It\^o isometry we obtain 
$$
\begin{aligned}
t\left\vert A_qP_tf(x)\right\vert &\le \left( \mathbb{E}\, f^2(L^x_t)\right)^{1/2}
\left(\int_0^t \int_H q^2(y)\nu(\dif y)\dif s\right)^{1/2}\\
&\le \left(P_tf^2(x)\right)^{1/2} t^{1/2} \left( \int_H q^2(y)\nu(\dif y)\right)^{1/2}.
\end{aligned}
$$
Thus the desired estimate holds for any $f\in UC_b^2(H)$. Assume that $f\in B_b(H)$. Let $x\in H$. Then there is a sequence $(f_n)\subset UC^2_b(H)$ such that 
$$
\lim_{n\to \infty} P_tf_n^2(x)=P_tf^2(x), 
$$
and 
$$
\lim_{n\to \infty} P_tf_n(x+y)=P_tf(x+y)\quad \text{for $\nu$ almost all  $y$.} 
$$
Consequently, the desired estimate for $f$ follows from the Fatou lemma. 
\end{proof}
\begin{corollary} \label{C32}For arbitrary  $f\in B_b(H)$ we have 
$$
\int_{H} \left\vert P_t f(x+y)-P_tf(x)\right\vert ^2 \nu(\dif y) \le  \frac 1tP_tf^2(x), \qquad \text{$x\in H$, $t>0$.}
$$
\end{corollary}
\begin{proof} Since 
$$
\int_{H} \left\vert P_tf(x+y)-P_tf(x)\right\vert ^2\nu(\dif y)= \sup \left\{ \left\vert A_q P_tf(x)\right\vert ^2, \quad q\colon \int_{H}q^2(y)\nu(\dif u)\le 1\right\}
$$
the estimate follows from Theorem \ref{T31}.
\end{proof}

Given  $f\in B_b(H)$ we  define the  difference  operator $\nabla^{n}_{y_1,\ldots, y_n}f(x)$, $x, y_1,\ldots,y_n\in H$ putting 
$$
\nabla _yf(x)= f(x+y)-f(x),
$$
$$
\nabla ^{n+1}_{y_1,\ldots,y_{n+1}}f(x)= \nabla _{y_{n+1}}\left(\nabla ^{n}_{y_1,\ldots,y_{n}}f\right)(x). 
$$
\begin{corollary}\label{C33}
For any $f\in B_b(H)$ and $n\in \mathbb{N}$, 
$$
\sup_{x\in H} \int_{H}\ldots\int_{H} \left\vert \nabla_{y_1,\ldots,y_n} \left(P_tf\right)(x)\right\vert ^2 \nu(\dif y_1)\ldots \nu(\dif y_n)\le \left(\frac{n}{t}\right)^n \|f\|_{\infty}^2. 
$$
\end{corollary}
\begin{proof} It is enough to show the  estimate for $f\in UC^2_b(H)$. Let $q_1,\ldots, q_n\in L^2(H,\mathcal B(H), \nu)$. We claim that the operators $A_{q}$ and $P_s$ commute. Indeed,  by the Feynman--Kac representation of $P_t$, the Fubini theorem and the fact
$X_t(x)+y=X_t(x+y)$, for all $f \in UC^2_b(\R^d)$ we have
\begin{align*}
P_t A_q f(x)&=\E \, A_q f(X_t(x))=
\E \, \int_{H} \left(f(X_t(x)+y)-f(X_t(x))\right)q(y)\nu(\dif y) \\
&=\int_{H} \left[\E f(X_t(x)+y)-\E f(X_t(x))\right]q(y)\nu(\dif y) \\
&=\int_{H} \left[P_tf(x+y)-P_tf(x)\right]q(y)\nu(\dif y)\\
&=A_qP_t f(x).
\end{align*}
Thus by Theorem \ref{T31}, 
\begin{align*}
\|A_{q_1}\ldots A_{q_n} P_tf\|^2_{\infty} &= \sup_{x\in H} \left\vert A_{q_1}\ldots A_{q_n} P_tf(x)\right\vert ^2 \\
&= \sup_{x\in H} \left\vert (A_{q_1}P_{t/n})\ldots (A_{q_n} P_{t/n})f(x)\right\vert ^2 \\
&\le \frac nt \|(A_{q_1}P_{t/n})\ldots (A_{q_n} P_{t/n})f\|^2_{\infty}. 
\end{align*}
\end{proof}

\section{Application 1: Short time behaviour of the semigroup}
We shall study the fractional gradient estimate of $\alpha$-stable and truncated $\alpha$-stable processes. In this and next sections $H=\mathbb{R}^d$. The norm on $\mathbb{R}^d$ will be denoted by $\vert\cdot \vert$. 
\begin{theorem}
Let the L\'evy measure $\nu$ be of the form
$$
\nu(\dif x)=\frac{1}{|x|^{d+\alpha}} \mathbf{1}_{\{|x|<K\}}\dif x,
$$
where $\alpha \in (0,2]$ and $K \in (0, \infty]$. Then for any $\beta \in (\alpha/2, \alpha)$ we have
\begin{equation} \label{e:DelPtf}
\|(-\Delta)^{\frac{\alpha-\beta}2} P_t f\|_\infty \le  C(1+t^{-1/2}) \|f\|_\infty,\qquad \forall\, t>0, \ f\in B_b(\mathbb{R}^d), 
\end{equation}
where $C=C_{\alpha,\beta}$ only depends on $\alpha$ and $\beta$.
\end{theorem}
\begin{proof}
Without any loss of generality, we may assume that $K=1$. Choose $q$ such that
$$
q(y)=|y|^\beta \qquad \forall \, |y| \le 1, \qquad   \int_{\{|y|>1\}} q^2(y) \nu(\dif y)<\infty. 
$$
It is easy to see that 
\begin{equation} \label{e:q2}
\int_{\{|y| \le 1\}} q^2(y) \nu(\dif y)=\int_{\{|y| \le 1\}} |y|^{-d-\alpha+2 \beta} \dif y \le C_{\alpha, \beta},
\end{equation}
where $C_{\alpha,\beta}>0$ depends on $\alpha, \beta$. Therefore, $q \in L^2(\nu):= L^2(\mathbb{R}^d,\mathcal{B}(\mathbb{R}^d), \nu)$.

Observe that 
$$
A_q f(x)=\int_{\{|y| \le 1\}} \frac{f(x+y)-f(x)}{|y|^{\alpha+d-\beta}} \dif y+\int_{\{|y|>1\}} \left[f(x+y)-f(x)\right] q(y) \nu(\dif y)
$$
and
\begin{align*}
\left\vert (-\Delta)^{\frac{\alpha-\beta}2} f(x)\right\vert &=\left|\int_{\R^d} \frac{f(x+y)-f(x)}{|y|^{\alpha+d-\beta}} \dif y\right| \\
& \le \left|\int_{\{|y|>1\}} \frac{f(x+y)-f(x)}{|y|^{\alpha+d-\beta}} \dif y\right|+\left|\int_{\{|y| \le 1\}} \frac{f(x+y)-f(x)}{|y|^{\alpha+d-\beta}} \dif y\right|.
\end{align*}
It is easy to see that
$$
\left|\int_{\{|y|>1\}} \frac{f(x+y)-f(x)}{|y|^{\alpha+d-\beta}} \dif y\right| \le c_{\alpha,\beta} \|f\|_\infty,
$$
where $c_{\alpha,\beta}>0$ depends on $\alpha, \beta$, and that
\begin{align*}
\left|\int_{\{|y| \le 1\}} \frac{f(x+y)-f(x)}{|y|^{\alpha+d-\beta}} \dif y\right| & \le \left|A_q f(x)\right|+\left|\int_{\{|y|>1\}} \left[f(x+y)-f(x)\right] q(y) \nu(\dif y)\right| \\
& \le t^{-1/2} \|f\|_\infty \|q\|_{L^2(\nu)}+2 \|f\|_\infty \|q\|_{L^2(\nu)},
\end{align*}
where the last inequality follows from Theorem \ref{T31} and H\"{o}lder's inequality. Collecting the previous inequalities, we  get the desired one.
\end{proof}

\section{Application 2: Estimate for a perturbed dynamics studied in \cite{ChWa12}}
Let $X_t(x)$ be the value at $t$ of the solution to the following stochastic differential equation
\begin{equation} \label{e:PertEqn}
\dif X_t=b(X_{t-}) \dif Y_t+\dif Z_t, \qquad  X_0=x,
\end{equation}
where $b \in C^2_b(\mathbb{R}^d)$ and $Z_t$, $Y_t$ are both symmetric stable processes with the parameters $\alpha, \beta \in (0,2)$. Our result below is also true if $b$ is a bounded measurable function, to avoid the complicated differentiability issue and stress the idea, we assume $b \in UC^2_b(\R^d)$.

Eq. \eqref{e:PertEqn} in more general setting has been studied by Chen and Wang (\cite{ChWa12}) and has a unique weak solution.  Let 
$$
P_t f(x)=\E f(X_t(x)), \qquad \forall \, f \in B_b(\R^d),
$$
be the corresponding transition semigroup.  The semigroup is $C_0$ on the space $UC_b(\mathbb{R}^d)$.  Let $\mathcal L$ be the generator of $(P_t)$ considered on $UC_b(\mathbb{R}^d)$. It is well known that $UC^2_b(\R^d) \subset\text{\rm  Dom}\, (\mcl L)$,  and that for all $f \in UC^2_b(\R^d)$,  
$$
\mathcal Lf=-(-\Delta)^{\alpha/2}f-|b(x)|^{\beta} (-\Delta)^{\beta/2}f,
$$
and the following backward Kolmogorov equation holds 
\begin{equation} \label{e:PtEqn}
\partial _t P_t f=\mathcal L P_tf. 
\end{equation}
We shall use \eqref{e:DelPtf} to show some properties of the associated backward Kolmogorov equation.

\begin{theorem}
If $\beta \in (0,\alpha/2)$, then there exists a  $t_0\in (0,1)$ depending on $\alpha, \beta$ and $\|b\|_{\infty}$, such that for any  $f \in UC_b(\R^d)$,
\begin{equation}\label{E53}
\begin{split}
& \|(-\Delta)^{\beta/2} P_tf\|_\infty \le C_1 t^{-1/2}\|f\|_\infty, \qquad  t \le t_0, \\
& \|(-\Delta)^{\beta/2} P_tf\|_\infty \le  C_2 \|f\|_\infty, \qquad t > t_0,
\end{split}
\end{equation}
where $C_1, C_2$ depend on $\alpha$, $\beta$ and $t_0$.
\end{theorem}

\begin{proof}
Since $UC^2_b(\R^d)$ is dense in $UC_b(\R^d)$ and $(-\Delta)^{\beta/2}$ is closable,  it suffices to show \eqref{E53}  for $f \in UC^2_b(\R^d)$.

 For any $f \in UC^2_b(\R^d)$, define $P^0_t f(x)=\E[f(Z_t+x)]$. It satisfies
 \begin{equation}\label{e:P0tEqn}
 \p_t P^0_t f=-(-\Delta)^{\alpha/2} P^0_t f.
 \end{equation}
$(P^0_t)_{t \ge 0}$ can be extended to  a Markov $C_0$-semigroup on $UC_b(\R^d)$. Thanks to \eqref{e:PtEqn} and \eqref{e:P0tEqn}, using the classical Duhamel principle we obtain
\begin{equation} \label{e:DuHam}
P_t f(x)=P^0_t f(x)-\int_0^t P^0_{t-s} [|b|^\beta (-\Delta)^{\beta/2}  P_s f](x) \dif s.
\end{equation}
Since $b \in UC^2_b(\R^d)$ and $f \in UC^2_b(\R^d)$, $P_t f \in UC^2_b(\R^d)$ and $P^0_t f \in U C^2_b(\R^d)$ both hold. By \eqref{e:DelPtf}, we have
\begin{equation}
\|(-\Delta)^{\beta/2} P^0_t f\|_\infty \le C(1+t^{-1/2}) \|f\|_\infty \le C t^{-1/2} \|f\|_\infty,  \qquad  \forall \, t<1,
\end{equation}
which, together with \eqref{e:DuHam}, yields
\begin{equation*}
\begin{split}
\|(-\Delta)^{\beta/2} P_t f\|_\infty & \le Ct^{-1/2} \|f\|_\infty+C\int_0^t (t-s)^{-1/2} \|b\|^\beta_\infty \|(-\Delta)^{\beta/2} P_s f\|_\infty \dif s \\
& \le Ct^{-1/2} \|f\|_\infty+C\int_0^t (t-s)^{-1/2} \|b\|^\beta_\infty \|(-\Delta)^{\beta/2} P_s f\|_\infty \dif s \\
& =Ct^{-1/2} \|f\|_\infty+C\int_0^t (t-s)^{-1/2} s^{-1/2}\|b\|^\beta_\infty s^{1/2} \|(-\Delta)^{\beta/2} P_s f\|_\infty \dif s.
\end{split}
\end{equation*}
Define
$$
L_{T}:=\sup_{0 \le t \le T} t^{1/2} \|(-\Delta)^{\beta/2} P_t f\|_\infty
$$
with $T>0$ to be chosen later. From the previous inequality we have
\Be
\begin{split}
L_{T} & \le C  \|f\|_\infty+ C T^{\frac 12} \sup_{0 \le t \le T} \int_0^t s^{-\frac 12} (t-s)^{-\frac 12} \dif s L_{T} \\
&=C  \|f\|_\infty+ C B(3/2, 3/2) T^{\frac 12} \|b\|^\beta_\infty\, L_{T},
\end{split}
\Ee
where $B$ is the beta function. Choose a  $t_0 \in (0,1)$ (depending on $\alpha$, $\beta$, and $\|b\|^\beta_\infty$) such that $C B(3/2, 3/2) t_0^{\frac 12}\|b\|^\beta_\infty<\frac 12$, we obtain 
$$ 
L_{t_0} \le 2 C \|f\|_\infty.
$$
This immediately implies the first estimate in the theorem.

For the second estimate, taking $P_{t_0/2} f$ rather than $f$ as the initial data, by the same procedure as above we have
\Bes
\begin{split}
\|(-\Delta)^{\beta/2} P_tf\|_\infty & \le C_1 \left(t-\frac{t_0}2\right)^{-\frac 12}\|P_{t_0/2}f\|_\infty \\
& \le C_1 \left(t-\frac{t_0}2\right)^{-\frac 12}\|f\|_\infty, \qquad  \forall  \, t \in (t_0/2, 3t_0/2).
\end{split}
\Ees
Therefore
$$\|(-\Delta)^{\beta/2} P_tf\|_\infty \le C_1 \left(\frac{t_0}2\right)^{-\frac 12}\|f\|_\infty,  \qquad  \forall  \,  t \in (t_0, 3t_0/2).$$
Now taking $P_{t_0} f$ as the initial data, we obtain 
$$
\|(-\Delta)^{\beta/2} P_tf\|_\infty \le C_1 \left(\frac{t_0}2\right)^{-\frac 12}\|f\|_\infty,  \qquad  \forall  \, t \in (3t_0/2,2t_0).
$$
Iterating the above argument, we finally get
$$
\|(-\Delta)^{\beta/2} P_tf\|_\infty \le C_1 \left(\frac{t_0}2\right)^{-\frac 12}\|f\|_\infty, \qquad  \forall \, t \ge t_0,
$$
which is the desired second estimate.
\end{proof}


\section{Application 3: modulus of continuity of transition semugroup of  $log$ $\alpha$-stable processes}
When the L\'evy measure $\nu$ on $\mathbb{R}^d$  satisfies
\begin{equation} \label{e:InfSmaJum}
\int_{\{|x|<r\}} \nu(\dif x)=\infty, \qquad  \forall \,  r>0,
\end{equation}
the process $(L_t)_{t \ge 0}$ has infinite small jumps in any time interval $[t,t+\delta)$. Then (see e.g. \cite{Sa90}), the law $\mathfrak{L}(L_t)$ of $L_t$ is continuous but not necessarily absolutely continuous with respect to Lebesgue measure. However, if additionally $\nu$ is absolutely continuous then $\mathfrak{L}(L_t)$ is absolutely continuous. Let us also recall, see e.g.  \cite{Hawkes}, that the law $\mathfrak{L}(L_t)$ is absolutely continuous if and only if for the corresponding semigroup satisfies $P_t\colon B_b(\mathbb{R}^d)\mapsto C_b(\mathbb{R}^d)$. Thus  the absolute continuity is equivalent to  strong Feller property. 

Below we provide some estimates for the moduli of continuity of the transition semigroup which is beyond the scope of $\alpha$-stable type process. Namely, assume that the L\'evy measure $\nu$ is absolutely continuous with respect to Lebesgue measure on the ball $B_1(0)=\{x\colon |x|<1\}$ and
\begin{equation} \label{e:NewNu}
\frac{\nu(\dif x)}{\dif x} \ge \frac{|\log_2|x||^{2\alpha}}{|x|^d}, \qquad  x \in B_1(0),
\end{equation}
where $\alpha \in (1,\infty)$ is a constant.

\begin{thm} \label{t:SStrFel}
Let $(L_t)_{t \ge 0}$ be a L\'evy process with L\'evy measure $\nu$ satisfying
\eqref{e:NewNu}. Then $(L_t)_{t \ge 0}$ is strong Feller. Moreover, there exists an  $r_0>0$ such that
\begin{equation}
|P_t f(x)-P_t f(y)| \le \frac{C}{|\log_2 |x-y||^{\alpha-1}}, \qquad |x-y| \le r_0,
\end{equation}
where $C$ depends on $\alpha, d, t$ and $r_0$.
\end{thm}

Let $\Omega \subset \R^d$ be an open set.  Given a function $f\colon \Omega \to \mathbb{R}$, $x\in \Omega$ and $r>0$ such that the ball $B_r(x)\subset \Omega$, define 
$$
\bar f_{x,r}:=\frac{1}{|B_r(x)|} \int_{B_r(x)} f(y) \dif y. 
$$ 
To show the theorem, we need to use the following lemma, which is a generalized Campanato theorem. The proof of the  lemma  is postponed to the appendix. 
\begin{lem}
\label{t:GMT}
Let $\Omega \subset \R^d$ be open and bounded, and let  $f\colon\Omega \rightarrow \R$ be  a bounded function. Assume that there are constants $C>0$  and  $\alpha >1$  such that 
\Be \label{e:MorCon}
\int_{B_r(x)} |f(y)-\bar f_{x,r}|^2 \dif y \le C \frac{r^{d}}{|\log_2 r|^{2\alpha}},
\Ee
for any ball $B_r(x) \subset \Omega$. Then $f$ is uniformly continuous and there exists an  $r_0>0$ such that for any open $\tl \Omega  \subset \Omega$ with $\text{\rm diam}\,(\tl \Omega)<r_0$ and $\text{\rm dist}\, (\tl \Omega, \p \Omega)>r_0$, we have
\begin{equation} \label{e:MorRes}
\sup_{x,y \in \tl \Omega, x \ne y} |f(x)-f(y)| |\log_2 |x-y||^{\alpha-1} \le \tilde C,
\end{equation}
where $\tilde C$ depends on $C$, $\text{\rm dist}\,(\tl \Omega, \p \Omega), d, \alpha$ and $r_0$.
\end{lem}

\begin{proof} [Proof of Theorem \ref{t:SStrFel}]
Let $f \in B_b(\R^d)$ and set $g(x)=P_t f(x)$. For any $t>0$, we have
\begin{equation} \label{e:MorPro}
\begin{split}
\int_{B_r(x)} |g(y)-\bar g_{x,r}|^2 \dif y
&=\int_{\{|y-x| \le r\}} \frac{1}{|B_r(x)|^2}\left|\int_{\{|z-x| \le r\}} \left(g(y)-g(z)\right) \dif z\right|^2 \dif y \\
& \le \int_{\{|y-x| \le r\}} \frac{1}{|B_r(x)|} \int_{\{|z-x| \le r\}} \left|g(y)-g(z)\right|^2 \dif z \dif y  \\
& \le \int_{\{|y-x| \le r\}} \frac{1}{|B_r(x)|} \int_{\{|y-z| \le 2 r\}} \left|g(y)-g(z)\right|^2  \dif z \dif y,
\end{split}
\end{equation}
where the last inequality follows from  the inclusion 
$$
\{|y-x| \le r\} \cap \{|z-x| \le r\} \subset \{|y-x| \le r\} \cap \{|y-z| \le 2r\}.
$$
From \eqref{e:MorPro} we further get
\begin{align*}
& \ \ \ \int_{B_r(x)} |g(y)-\bar g_{x,r}|^2 \dif y \\
& \le \int_{\{|y-x| \le r\}} \frac{1}{|B_r(x)|} \int_{\{|y-z| \le 2 r\}} \left|g(y)-g(z)\right|^2
\frac{|z-y|^{d}}{|\log_2 |y-z||^{2\alpha}} \frac{|\log_2 |z-y||^{2\alpha}}{|z-y|^{d}}\dif z \dif y \\
&=\int_{\{|y-x| \le r\}} \frac{1}{|B_r(x)|} \int_{\{|z| \le 2 r\}} \left|g(y+z)-g(y)\right|^2
\frac{|z|^{d}}{|\log_2 |z||^{2\alpha}} \frac{|\log_2 |z||^{2\alpha}}{|z|^{d}}\dif z \dif y  \\
& \le \int_{\{|y-x| \le r\}} \frac{1}{|B_r(x)|} \int_{\{|z| \le 2 r\}} \left|g(y+z)-g(y)\right|^2  \frac{|z|^{d}}
{|\log_2 |z||^{2\alpha}} \nu(\dif z) \dif y
\end{align*}
Since $\frac{r^{d}}
{|\log_2 r|^{2\alpha}}$ is decreasing as $r<r_0/2$ for small $r_0>0$, the above inequality further gives
\Bes
\begin{split}
\int_{B_r(x)} |g(y)-\bar g_{x,r}|^2 \dif y & \le C\frac{r^{d}}
{|\log_2 r|^{2\alpha}}  \frac{1}{|B_r(x)|} \int_{\{|y-x| \le r\}}\int_{\{|z| \le 2 r\}} \left|g(y+z)-g(y)\right|^2  \nu(\dif z) \dif y \\
& \le C\frac{r^{d}}
{|\log_2 r|^{2\alpha}}  \sup_{y \in B_r(x)}  \int_{\R^d} \left|g(y+z)-g(y)\right|^2  \nu(\dif z) \\
&  \le C\frac{r^{d}}
{|\log_2 r|^{2\alpha}} t^{-1} \|f\|^2_\infty,
\end{split}
\Ees
where the last inequality follows from  Corollary \ref{C32}. Combining the estimate above and  Lemma \ref{t:GMT} we obtain the desired conclusion.
\end{proof}

\section{Appendix: Proof of Lemma \ref{t:GMT}}
 We shall follow \cite{HaLi11}. For $0<r_2<r_1<\min\{\text{\rm dist}\, (\tl \Omega, \p \Omega),1\}$ and any $x \in \tl \Omega$, we have
\begin{equation} \label{e:CaufBar}
\begin{split}
|\bar f_{x,r_1}-\bar f_{x,r_2}| &\le \frac{1}{|B_{r_2}|} \int_{B_{r_2}(x)} |f(y)-\bar f_{x,r_2}| \dif y+\frac{1}{|B_{r_2}|} \int_{B_{r_2}(x)} |f(y)-\bar f_{x,r_1}| \dif y \\
& \le \sqrt{\frac{1}{|B_{r_2}|} \int_{B_{r_2}(x)} |f(y)-\bar f_{x,r_2}|^2 \dif y}+\sqrt{\frac{1}{|B_{r_2}|} \int_{B_{r_2}(x)} |f(y)-\bar f_{x,r_1}|^2 \dif y} \\
&\le \sqrt{\frac{1}{|B_{r_2}|} \int_{B_{r_2}(x)} |f(y)-\bar f_{x,r_2}|^2 \dif y}+\sqrt{\frac{|B_{r_1}|}{|B_{r_2}|}\frac{1}{|B_{r_1}|} \int_{B_{r_1}(x)} |f(y)-\bar f_{x,r_1}|^2 \dif y} \\
&\le C\left[\frac{1}{|\log_2 r_2|^{\alpha}}+\left(\frac{r_1}{r_2}\right)^{d/2} \frac{1}{|\log_2 r_1|^\alpha}\right] \\
& \le C\left[1+\left(\frac{r_1}{r_2}\right)^{d/2}\right] \frac{1}{|\log_2 r_1|^\alpha},
\end{split}
\end{equation}
where the last inequality is by the assumption of  the lemma. For all $0<r_n<r_m<\text{\rm dist}\,(\tl \Omega, \p \Omega)$, define
$$N:=\left[\log_2 \left(\frac{r_m}{r_n}\right)\right],$$
without loss of generality we assume $r_m<1/2$.
By \eqref{e:CaufBar} we have
\begin{equation}
\begin{split}
|\bar f_{x,r_n}-\bar f_{x,r_m}| &\le \sum_{k=1}^{N} |\bar f_{x,2^{-k} r_m}-\bar f_{x,2^{-k+1}r_m}|+|\bar f_{x,2^{-N} r_m}-\bar f_{x,r_n}| \\
& \le C_d \sum_{k=1}^{N} \frac{1}{|k-1-\log_2 r_m|^\alpha}+C_d\frac{1}{|N-\log_2 r_m|^\alpha} \\
& \le C_{d,\alpha} \frac{1}{|\log_2 r_m|^{\alpha-1}}.
\end{split}
\end{equation}
Hence, there exists an  $\tl f$ such that
$$
\lim_{r \rightarrow 0} \bar f_{x,r}=\tl f(x), \qquad  \forall \, x \in \tl\Omega,
$$
and there exists an  $r_0>0$ such that as $r<r_0$,
\begin{equation}  \label{e:MeaCon}
|\bar f_{x,r}-\tl f(x)| \le C_{d,\alpha} \frac{1}{|\log_2 r|^{\alpha-1}}, \qquad \forall \, x \in \tl\Omega.
\end{equation}
On the other hand, by the Lebesgue theorem,
$$
\lim_{r \rightarrow 0} \bar f_{x,r}=f(x), \qquad \text{a.s.} \ x \in \tl \Omega.
$$
By \eqref{e:MeaCon}, all the points in $\tl \Omega$ are Lebesgue points. Hence, $\bar f_{x,r} \rightarrow f(x)$ uniformly for $x \in \tl\Omega$ as $r \rightarrow 0$ with
\begin{equation}  \label{e:MeaCon1}
|\bar f_{x,r}-f(x)| \le C_{d,\alpha} \frac{1}{|\log_2 r|^{\alpha-1}}, \qquad  \forall \,   x \in \tl\Omega.
\end{equation}
Now for $x, y \in \tl \Omega$, denote $r=|x-y|$, we have
\begin{equation}
|\bar f_{y,2r}-\bar f_{x,2r}| \le  |\bar f_{y,2r}-f(z)|+|f(z)-\bar f_{x,2r}|, \qquad  z \in B_{2r}(x) \cap B_{2r}(y).
\end{equation}
Since $B_{2r}(x) \cap B_{2r}(y)$ contains a ball with radius $r$, as $r<r_0/2$,
\begin{equation} \label{e:fBarCau}
\begin{split}
|\bar f_{y,2r}-\bar f_{x,2r}| & \le  \frac{1}{|B_r|} \int_{B_{2r}(x) \cap B_{2r}(y)} |\bar f_{y,2r}-f(z)|+|f(z)-\bar f_{x,2r}| \dif z  \\
& \le \frac{1}{|B_r|} \int_{B_{2r}(y)} |\bar f_{y,2r}-f(z)| \dif z+\frac{1}{|B_r|} \int_{B_{2r}(x)}|f(z)-\bar f_{x,2r}| \dif z \\
& \le C \frac{1}{|\log_2 r|^{\alpha}},
\end{split}
\end{equation}
where the last inequality is by the assumption of  the lemma. Observe that 
\begin{equation}
|f(x)-f(y)| \le |f(x)-\bar f_{x,2r}|+|f(y)-\bar f_{y,2r}|+|\bar f_{y,2r}-\bar f_{x,2r}|.
\end{equation}
This, together with \eqref{e:fBarCau} and \eqref{e:MeaCon1}, immediately implies the desired inequality.

\bibliographystyle{amsplain}

\end{document}